\documentclass[onefignum,onetabnum]{siamart190516}

\usepackage{braket,amsfonts}

\usepackage{array}

\usepackage[caption=false]{subfig}
\usepackage{pgfplots}

\newsiamthm{claim}{Claim}
\newsiamremark{remark}{Remark}
\newsiamremark{hypothesis}{Hypothesis}
\crefname{hypothesis}{Hypothesis}{Hypotheses}

\usepackage{algorithmic}

\usepackage{graphicx,epstopdf}

\Crefname{ALC@unique}{Line}{Lines}

\usepackage{amsopn}

\usepackage{xspace}
\usepackage{bold-extra}
\usepackage[most]{tcolorbox}

\colorlet{texcscolor}{blue!50!black}
\colorlet{texemcolor}{red!70!black}
\colorlet{texpreamble}{red!70!black}
\colorlet{codebackground}{black!25!white!25}

\lstdefinestyle{siamlatex}{%
	style=tcblatex,
	texcsstyle=*\color{texcscolor},
	texcsstyle=[2]\color{texemcolor},
	keywordstyle=[2]\color{texemcolor},
	moretexcs={cref,Cref,maketitle,mathcal,text,headers,email,url},
}

\tcbset{%
	colframe=black!75!white!75,
	coltitle=white,
	colback=codebackground, % bottom/left side
	colbacklower=white, % top/right side
	fonttitle=\bfseries,
	arc=0pt,outer arc=0pt,
	top=1pt,bottom=1pt,left=1mm,right=1mm,middle=1mm,boxsep=1mm,
	leftrule=0.3mm,rightrule=0.3mm,toprule=0.3mm,bottomrule=0.3mm,
	listing options={style=siamlatex}
}

\newtcblisting[use counter=example]{example}[2][]{%
	title={Example~\thetcbcounter: #2},#1}

\newtcbinputlisting[use counter=example]{\examplefile}[3][]{%
	title={Example~\thetcbcounter: #2},listing file={#3},#1}

\DeclareTotalTCBox{\code}{ v O{} }
{ %fontupper=\ttfamily\color{texemcolor},
	fontupper=\ttfamily\color{black},
	nobeforeafter,
	tcbox raise base,
	colback=codebackground,colframe=white,
	top=0pt,bottom=0pt,left=0mm,right=0mm,
	leftrule=0pt,rightrule=0pt,toprule=0mm,bottomrule=0mm,
	boxsep=0.5mm,
	#2}{#1}

\patchcmd\newpage{\vfil}{}{}{}
\usepackage{epsf,latexsym,amsfonts,amssymb,mathrsfs,graphicx,stmaryrd}
\usepackage{color}
\newcommand{\nn}{\nonumber}
\newtheorem{coro}[theorem]{Corollary}

\theoremstyle{definition}

\numberwithin{equation}{section}
\newcommand{\abs}[1]{\left\lvert#1\right\rvert}
\newcommand{\diff}[2]{\dfrac{\partial #1}{\partial #2}}
\newcommand{\Lr}[1]{\left(\,#1\,\right)}
\newcommand{\mc}[1]{\mathcal #1}
\newcommand{\mb}[1]{\mathbb #1}
\newcommand{\nm}[2]{\left\|\,#1\,\right\|_{#2}}
\newcommand{\wt}[1]{\widetilde{#1}}
\newcommand{\wh}[1]{\widehat{#1}}

\renewcommand{\set}[2]{\{\,#1\,\mid\,#2\,\}}
\newcommand{\dual}[2]{\left\langle #1,#2\right\rangle}
\def\negint{{\int\negthickspace\negthickspace\negthickspace
		\negthinspace -}}

\def\al{\alpha}
\def\del{\delta}
\def\eps{\varepsilon}
\def\xe{(x/\varepsilon)}
\def\xxe{\Lr{\dfrac{x}\eps}}
\def\na{\nabla}
\def\pa{\partial}
\def\lam{\lambda}
\def\Lam{\Lambda}

\def\Om{\Omega}
\def\R{\mb{R}}
\def\dx{\,\mathrm{d}x}
\def\dy{\,\mathrm{d}y}
\def\x{\times}
\def\divop{\operatorname{div}}

\begin{tcbverbatimwrite}{tmp_\jobname_header.tex}
\title{Error Estimate of Multiscale Finite Element Method for Periodic Media Revisited
\thanks{Submitted to the editors \today.
\funding{The work of Ming was supported by the National Natural Science Foundation of China under the grants 11971467 and 12371438.}}
}%
\author{Pingbing Ming\thanks{LSEC, Institute of Computational Mathematics and Scientific/Engineering Computing, AMSS,
			Chinese Academy of Sciences, No. 55, East Road Zhong-Guan-Cun, Beijing 100190, China
			and School of Mathematical Sciences, University of Chinese Academy of Sciences, Beijing 100049, China (\email{mpb@lsec.cc.ac.cn},\email{songsq@lsec.cc.ac.cn}).}
		\and Siqi Song\footnotemark[2]}
	
\headers{Error Estimates of MsFEM}{Pingbing Ming, and Siqi Song}
\end{tcbverbatimwrite}
\input{tmp_\jobname_header.tex}

\ifpdf
\hypersetup{ pdftitle={Error Estimate of Multiscale Finite Element Method for Periodic Media Revisited}}
\fi
\begin{document}
	\maketitle
\begin{tcbverbatimwrite}{tmp_\jobname_abstract.tex}
\begin{abstract}
We derive the optimal energy error estimate for multiscale finite element method with oversampling technique applying to elliptic systems with rapidly oscillating periodic coefficients that are bounded measurable, which may admit rough microstructures. As a by-product of the energy error estimate, we derive the rate of convergence in 
L$^{d/(d-1)}-$norm with $d$ the dimensionality.
\end{abstract}

\begin{keywords}
Multiscale finite element method, homogenization, error estimate, oversampling
\end{keywords}
\begin{AMS}
35J15, 65N12, 65N30
\end{AMS}
\end{tcbverbatimwrite}
\input{tmp_\jobname_abstract.tex}
%% ------------------------------------------------------------------
%% END HEADER
%% ------------------------------------------------------------------

%\maketitle
%
\section{Introduction}
The multiscale finite element method (MsFEM) introduced by Hou an Wu~\cite{HouWu:1997} aims for solving the boundary value problems with rapidly oscillating coefficients without resolving the fine scale information. The main idea is to exploit the multiscale basis functions that capture the fine scale information of the underlying partial differential equations. MsFEM has been successfully applied to many problems such as two phase flows, nonlinear homogenization problems, convection-diffusion problems,  elliptic interface problems with high-contrast coefficients and Poisson problem with rough and oscillating boundary, we refer to book~\cite{EfendievHou:2009} for a survey of MsFEM before 2009.  More recent efforts for MsFEM focus on extending the method to deal with more general media; cf.,~\cite{ChungEfendiev:2014, Chen:2021, Chen:2022}. We also refer 
to~\cite{Owhadi:2017, Owhadi:2019, peterseim:2021,ChenChenLi:2022} for a summary of recent progress for related 
methods.

%Such ideas may be traced back to the generalized finite element method in~\cite{BabuskaOsborn:1983}. 
In~\cite{HouWuCai:1999} and~\cite{EfendievHouWu:2000}, the authors proved MsFEM converges for the scalar elliptic boundary value problem in two dimension with periodic oscillating coefficients in the energy norm, and the convergence rate is $\sqrt{\eps}+h+\eps/h$, where $h$ is the mesh size of the triangulation, and $\eps$ is the period of the oscillation. The technical assumptions are%in their work~\cite{HouWuCai:1999, EfendievHouWu:2000}: 
\begin{enumerate}
\item The coefficient matrix of the elliptical problem is symmetric, and each entry is a $C^1$ function;

\item The homogenized solution $u_0\in W^{1,\infty}(\Om)\cap H^2(\Om)$;

\item The corrector $\chi$ defined in~\eqref{eq:corrector} belong to $W^{1,\infty}$.
\end{enumerate}
The first assumption excludes the rough microstructures, which frequently appears in the realistic materials~\cite{Torquato:2002}; The second assumption is standard except that $u_0\in W^{1,\infty}(\Om)$, which may not be true even for Poisson equation posed on a ball~\cite{Cianchi:1990}. The last assumption on the corrector is not realistic at all, though it may be true for certain special microstructures such as laminates~\cite{ChipotKinderlehrer:1986} and for problems with piecewise H\"{o}lder continuous coefficients~\cite{LiVogelius:2000,LiNirenberg:2003}; We refer to~\cite{DuMing:2010} for an elaboration on this assumption.

Nevertheless, there are some subsequent endeavor on proving the error estimates for MsFEM under weaker assumptions; see, e.g.,~\cite{ChenHou:2003,SarkisVersieux:2008,ChenWu:2010, YeDongCui:2020}, just name a few, most of them concern the second assumption, while it is still unknown whether the above assumptions may be removed or to what degree they may be weakened. Moreover, though MsFEM has been successfully applied to elliptic systems~\cite{EfendievHou:2009,ChungEfendiev:2014}, while it does not seem easy to extend the proof to elliptic systems because the maximum principle has been exploited, which may be invalid for elliptic systems~\cite{KresinMazya:2012}. 

The present work gives an affirmative answer to the above questions. Assuming that $u_0\in W^{2,d}$ with the dimensionality $d=2,3$, we prove the optimal energy error estimate of MsFEM with/without oversampling for elliptic systems with bounded, measurable and symmetric periodical coefficients; cf. Theorem~\ref{thm:main} and Theorem~\ref{thm:mainorig}. The symmetry assumption may be dropped for MsFEM without oversampling, or for MsFEM with oversampling applying to the elliptic scalar problem. This means that MsFEM achieves optimal convergence rate for problems with rough microstructures. %Furthermore, the same rate of convergence is proved for the elliptic scalar problem {\em without} the symmetrical assumption on the coefficient matrix. %The practical implication of such result is that the methods tailored to resolve the underlying rough microstructure are in great demand to improve the overall accuracy of MsFEM. %We refer to [16, 2, 6, 25] and the references therein for such methods.

As an application of the energy error estimate, we derive improved error estimate of MsFEM in $L^{d/(d-1)}-$norm by resorting to the Aubin-Nitsche dual argument~\cite{Aubin:1967, Nitsche:1968}, naturally, this gives the $L^2-$error estimates for two-dimensional problem and the elliptic scalar problem in three dimension. Such estimate would be useful for analyzing MsFEM applying to the eigenvalue problems in composites~\cite{Kesavan:1979}.

There are two ingredients in our proof. The one is a local version of the multiplier estimates for periodic homogenization of elliptic systems~\cite{Zhikov:2005, Shen:2018}; see Lemma~\ref{lemma:Mulplier}, which helps us to remove the boundedness assumption on the gradient of the corrector. Another one is a local estimate of the gradient of the first order approximation of the solution; see Lemma~\ref{lemma:local1st}, which bypasses the maximum principle in the proof, hence we may derive the error estimate for elliptic systems.

The remaining part of the paper is as follows. We formulate MsFEM with oversampling in \S~\ref{sec:MsFEM}. In \S~\ref{sec:homogenization}, we recall some quantitative estimates of the periodic homogenization for elliptic systems. The energy error estimate will be given in \S~\ref{sec:error}, from which we prove the error estimates in $L^{d/d-1}$ norm.  As a direct consequence of these estimates, we prove the error estimates for MsFEM without oversampling. In the last section, we summarize our results and discuss certain extensions.

Throughout this paper,  $C$ is a generic constant that may be different at different occurrence, while it is independent of the mesh size $h$ and the small parameter $\eps$.
\section{Multiscale Finite Element Method with Oversampling}\label{sec:MsFEM}
We firstly fix some notations. Let $\Om$ be a bounded Lipschitz domain in $\mathbb{R}^d$ (we focus on $d=2,3$). The standard Sobolev space $W^{k,p}(\Om)$ will be used~\cite{AdamsFournier:2003}, which is equipped with the norm $\nm{\cdot}{W^{k,p}(\Om)}$. We use the convention $H^k(\Om)=W^{k,2}(\Om)$. We denote by $W^{k,p}(\Om;\R^m)$ the vector-valued function with each component belonging to $W^{k,p}(\Om)$, and define $\abs{D}{:}=\text{mes} D$ for any measurable set $D$. 

We consider the second order elliptic system in divergence form
\[
\mc{L}_{\eps}=-\divop\Lr{A\xe\na}
\]
with the coefficient $A$ given by
\[
A(y)=a_{ij}^{\al\beta}(y)\qquad i,j=1,\cdots,d\text{\;and\;}\al,\beta=1,\cdots,m.
\]
For $u=(u^1,\cdots,u^m)$,
\[
\Lr{\mc{L}_{\eps}(u)}^{\al}{:}=-\diff{}{x_i}\Lr{a_{ij}^{\al\beta}\xxe\diff{u^{\beta}}{x_j}}\qquad\al=1,\cdots,m.
\]
We always assume that $A$ is bounded measurable and satisfies the Legendre-Hadamard condition as
\begin{equation}\label{eq:ellp}
	\lam\abs{\xi}^2\abs{\eta}^2\le a_{ij}^{\al\beta}(y)\xi_i\xi_j\eta_{\al}\eta_{\beta}\le\Lam\abs{\xi}^2\abs{\eta}^2
	\qquad\text{for a.e.\;}y\in\mb{R}^d,
\end{equation}
where $\xi=(\xi_1,\cdots,\xi_d)$ and $\eta=(\eta_1,\cdots,\eta_m)$. The transpose of $A$ is understood as $A^t(y)=a_{ji}^{\beta\al}(y)$. We assume that $A$ is $1-$periodic; i.e., for all $z\in\mb{Z}^d$,
\[
A(y+z)=A(y)\qquad\text{for a.e.\;}y\in\mb{R}^d.
\]

Considering the following homogeneous boundary value problem: Given $f\in H^{-1}(\Om;\mb{R}^m)$, we find $u^\eps
\in H_0^1(\Om;\R^m)$ satisfying
\begin{equation}\label{eq:bvp}
\mc{L}_{\eps}(u^{\eps})=f\quad\text{in\quad}\Om\qquad\text{and\quad}u^{\eps}=0\quad\text{on\quad}\pa\Om
\end{equation}
in the sense of distribution. The corresponding variational problem reads as: Find $u^\eps\in H_0^1(\Omega;\mathbb{R}^m)$ such that
\begin{equation}\label{eq:bvpweak}
a_{{\Om}}(u^\eps,v)=\dual{f}{v}_{\Om}\quad \text{for all~} v\in H_0^1(\Omega;\mathbb{R}^m),
\end{equation}
where for any measurable subset $\wt{\Om}$ of $\Omega$, 
\[
a_{\wt{\Om}}(u,v){:}=\int_{\wt{\Om}}\na v\cdot A(x/\eps)\na u\dx\qquad\text{and\quad}\dual{f}{v}_{\wt{\Om}}=\int_{\wt{\Om}} f(x)\cdot v(x)\dx.
\]
We shall drop the subscript when the subset is the whole domain $\Om$.
%It follows from the Lax-Milgram theorem that Problem~\eqref{eq:bvpweak} has a unique solution.

$\Omega$ is triangulated by $\mc{T}_h$ that consists of simplices $\tau$ with $h_{\tau}$ its diameter and $h=\max_{\tau\in\mc{T}_h}h_{\tau}$. We assume that $\mc{T}_h$ is shape-regular in the sense of Ciarlet-Raviart~\cite{Ciarlet:1978}: there exists a chunkiness parameter $\sigma_0$ such that \(h_{\tau}/\rho_\tau\le\sigma_0\),
where $\rho_{\tau}$ is the diameter of the largest ball inscribed into $\tau$. We also assume that $\mc{T}_h$ satisfies the inverse assumption: there exists $\sigma_1>0$ such that \(h/h_{\tau}\le\sigma_1\).% for all $\tau\in\mc{T}_h$.

For each element $\tau$, we firstly choose an oversampling domain $S=S(\tau)\supset\tau$,  which is also a simplex. Let $\lam_i$ be the $i$th barycentric coordinate of the simplex $S$ and $e^{\beta}=(0,\cdots,1\cdots,0)$ with $1$ in the
$\beta$th position. Denote $Q\in\R^{(d+1)\x m}$ with $Q_i^{\beta}=\lam_ie^\beta$ for $i=1,\cdots,d+1$ and $\beta=1,\cdots,m$, we find $\psi_i^\beta-Q_i^{\beta}\in H_0^1(S;\mb{R}^m)$ such that
\begin{equation}\label{eq:overpro}
	a_S(\psi_i^\beta,\varphi)=0\quad \text{for all\quad}\varphi\in H_0^1(S;\mb{R}^m).
\end{equation}

Next, the basis function  $\phi_i^{\beta}$ associated with the node $x_i$ of $\tau$ is defined as
%we construct the basis function on $\tau$ as the linear combination of $\psi_i^\beta$. Let $\phi_i^\beta$ be the basis function associating with the node $x_i$ and
\begin{equation}\label{eq:basis}
	\phi_i^{\beta}= c_{ij}^\beta\psi_j^{\beta}\qquad i=1,\cdots, d+1\quad\text{and\;}\beta=1,\cdots,m,
\end{equation}
where the coefficients $c_{ij}^\beta$ are determined by $c_{ik}^\beta Q_k^{\beta}(x_j)=\delta_{ij}e^{\beta}$ for any node
$x_j$ of $\tau$. The matrix $c^\beta=(c_{ij}^\beta)$ is invertible because $\{\psi_i^\beta\}_{i=1}^{d+1}$ are linear independent over $S$. For $\phi_i=(\phi_i^1,\phi_i^2,\cdots,\phi_i^m)$, the multiscale finite element space is defined by
\[
V_h:=\text{Span}\{\phi_i \quad\text{for all nodes\;} x_i \text{\;of\;}\mc{T}_h\}.
\]
Note that $V_h\subsetneq H^1(\Om;\R^m)$ because the functions in $V_h$ may not be continuous across the element boundary. The bilinear form $a_h$ is defined for any $v,w\in V_h$ in a piecewise manner as 
\(
a_h(v,w):=\sum_{\tau\in\mc{T}_h} a_{\tau}(v,w).
\)
The approximation problem reads as: Find $u_h\in V_h^0$ such that
\begin{equation}\label{eq:overMsFEM}
a_h(u_h,v)=\dual{f}{v}\qquad\text{for all\quad}v\in V_h^0,
\end{equation}
where
\(
V_h^0{:}=\{v\in V_h|\  \text{the degrees of freedom of the nodes on~}\pa \Omega\text{~are zero}\}\). It follows from~\cite[Appendix B]{EfendievHouWu:2000} that 
\begin{equation}\label{eq:norm}
\nm{v}{h}:=\Lr{\sum_{\tau\in\mc{T}_h}\|\na v\|_{L^2(\tau)}^2}^{1/2}
\end{equation}
is a norm over $V_h^0$. 
\begin{remark}
The authors in~\cite{HenningPeterseim:2013} introduced a new MsFEM that allows for the oversampling domain of more general shape, e.g. an element star, which facilitates the implementation of MsFEM, while it is equivalent to the original version~\cite{EfendievHouWu:2000} if the oversampling domain is a simplex.
\end{remark}
\section{Quantitative Estimates for Periodic Homogenization of Elliptic System}\label{sec:homogenization}
By the theory of H-convergence~\cite{MuratTartar:1997}, the solution $u^\eps$ of~\eqref{eq:bvp} converges weakly to the homogenized solution $u_0$ in $H^1(\Omega;\mb{R}^m)$ as $\eps\to 0$, and $u_0$ satisfies
\begin{equation}\label{eq:bvp0}
\mc{L}_0(u_0)=f\quad\text{in\quad}\Om,\qquad u_0=0\quad\text{on\quad}\pa\Om,
\end{equation}
where
\(
\mc{L}_0=\divop(\wh{A}\na)
\)
with the homogenized coefficients $\wh{A}=\wh{a}_{ij}^{\al\beta}$ given by
\[
\wh{a}_{ij}^{\al\beta}=\negint_Y\Lr{a_{ij}^{\al\beta}(y)+a_{ik}^{\al\gamma}\diff{\chi_j^{\gamma\beta}}{y_k}}\dy,
\]
where the unit cell $Y{:}=[0,1)^d$, and the corrector $\chi(y)=\Lr{\chi_j^{\beta}(y)}=\Lr{\chi_j^{\al\beta}}$ for $j=1,\cdots,d$ and $\al,\beta=1,\cdots,m$ satisfies the following cell problem: Find $\chi_j^{\beta}\in H^1_{\text{per}}(Y;\mb{R}^m)$ such that $\int_Y\chi_j^{\beta}\dy=0$ and
\begin{equation}\label{eq:corrector}
a_Y(\chi_j^{\beta},\psi)=-a_Y(P_j^{\beta},\psi)\qquad\text{for all\quad}\psi\in H^1_{\text{per}}(Y;\mb{R}^m),
\end{equation}
where $P_j^{\beta}=y_j e^{\beta}$, and for all $\phi,\psi\in H^1_{\text{per}}(Y;\mb{R}^m)$,
\[
a_Y(\phi,\psi){:}=\int_Ya_{ij}^{\al\beta}(y)\diff{\phi^{\beta}}{y_j}\diff{\psi^{\al}}{y_i}\dy.
\]

The existence and uniqueness of the solution of~\eqref{eq:corrector} follows from the ellipticity of $A$ and the Lax-Milgram theorem. Moreover, 
\[
\nm{\na\chi_j^\beta}{L^2(Y)}\le\Lam/\lam\qquad\text{and}\quad\nm{\chi_j^{\beta}}{H^1(Y)}\le C_p\Lam/\lam,
\]
where $C_p$ is the constant arising from Poincar\'e's inequality:
\[
\nm{\psi}{H^1(Y)}\le C_p\nm{\na\psi}{L^2(Y)}\quad\text{for all\quad}\psi\in H^1_{\text{per}}(Y)\quad\text{and}\quad\int_Y\psi\dy=0.
\]

By Meyers' regularity result~\cite{Meyers:1963, Meyers:1975}, there exists $p>2$ such that
\begin{equation}\label{eq:meyers}
\nm{\na\chi_j^\beta}{L^p(Y)}\le C,%\qquad\text{for certain\quad}p>2,
\end{equation}
where the index $p$ and the constant $C$ depending only on $\lam$ and $\Lam$. This inequality implies that $\chi$ is H\"older continuous when $d=2$ by the Sobolev embedding theorem~\cite{AdamsFournier:2003}. By the De Giorgi-Nash theorem, $\chi$ is also H\"older continuous when $d=3$ and $m=1$. Hence, for $m=1,d=2,3$ and $m\ge 2, d=2$, there exists $C$ depending only on $\lam$ and $\Lam$ such that
\begin{equation}\label{eq:maxbd}
\nm{\chi_j^{\beta}}{L^\infty(Y)}\le C.
\end{equation}

In case of $d=3$ and $m\ge 2$, we only have
\begin{equation}\label{eq:corr-reg}
\nm{\chi_j^\beta}{L^q(Y)}\le C\qquad\text{for certain\quad}q\ge 6,
\end{equation}
which is a direct consequence of~\eqref{eq:meyers} and the Sobolev embedding theorem~\cite{AdamsFournier:2003}.

Another frequently used estimate for the corrector matrix is: For any measurable set $D$, and for $1\le p\le\infty$, there exists $C$ depends on $d$ and $p$ such that
\begin{equation}\label{eq:corr-est}
\nm{\chi(x/\eps)}{L^p(D)}\le C\abs{D}^{1/p}\nm{\chi}{L^p(Y)}.
\end{equation}%

Given the corrector $\chi$, the first order approximation of $u^\eps$ is defined by
\begin{equation}\label{eq:def1stapp}
u_1^{\eps}(x){:}=u_0(x)+\eps\chi\xe\na u_0(x).
\end{equation}
\iffalse which may be written into a component form as
\[
(u_1^{\eps}(x))^{\al}{:}=u_0^\al+\eps\chi_j^{\al\beta}\diff{u_0^{\beta}}{x_j}.
\]\fi

We summarize the convergence rate of $u_1^\eps$ in the following theorem. 
\begin{theorem}\label{thm:rate}
Assume that $A$ is $1-$periodic and satisfies~\eqref{eq:ellp}. Let $\Om$ be a bounded Lipschitz domain in $\mb{R}^d$. Let $u^{\eps}$ and $u_0$ be the weak solutions of~\eqref{eq:bvp} and ~\eqref{eq:bvp0}, respectively. 
\begin{enumerate}
\item If $u_0\in W^{2,d}(\Om;\mb{R}^m)$, then
\begin{equation}\label{eq:1stratea}
\nm{u^{\eps}-u_1^{\eps}}{H^1(\Om)}\le C\sqrt{\eps}\nm{\na u_0}{W^{1,d}(\Om)},
\end{equation}
where $C$ depends on $\lam,\Lam$ and $\Om$.

\item If the corrector $\chi$ is bounded and $u_0\in H^2(\Om;\mb{R}^m)$, then
\begin{equation}\label{eq:1strateb}
	\nm{u^{\eps}-u_1^{\eps}}{H^1(\Om)}\le C\sqrt{\eps}\nm{\na u_0}{H^1(\Om)},
\end{equation}
where $C$ depends $\lam,\Lam,\nm{\chi}{L^\infty}$ and $\Om$.\end{enumerate}
\end{theorem}

The estimates~\eqref{eq:1stratea} and~\eqref{eq:1strateb} are taken from~\cite[Theorem 3.2.7]{Shen:2018}.

We also need the following estimate in certain $L^p-$norm.
\begin{theorem}\label{thm:scale-invariant}
%Assume that $A$ is $1-$periodic and satisfies~\eqref{eq:ellp}. Let $\Om$ be a bounded Lipschitz domain in $\mb{R}^d$. Let $u^{\eps}$ and $u_0$ be the weak solutions of~\eqref{eq:bvp} and ~\eqref{eq:bvp0}, respectively. 
Under the same assumption of Theorem~\ref{thm:rate}, and assume that $A=A^t$ for $m\ge 2$. Suppose that $u_0\in W^{2,q}(\Om;\R^m)$ for $q=2d/(d+1)$. Then
\begin{equation}\label{eq:l2rate}
\nm{u^{\eps}-u_0}{L^p(\Om)}\le C\eps\nm{\na u_0}{W^{1,q}(\Om)},
\end{equation}
where $p=2d/(d-1)$ and $C$ depends only on $\lam,\Lam$ and $\Om$.
\end{theorem}
This theorem was proved in~\cite{Shen:2017}; See also~\cite[Theorem 3.4.3]{Shen:2018} with
\[
\nm{u^{\eps}-u_0}{L^p(\Om)}\le C\eps\nm{u_0}{W^{2,q}(\Om)},
\]
which together with the Ponicar\'e's inequality leads to~\eqref{eq:l2rate}. Moreover, using a scaling argument, we rewrite~\eqref{eq:l2rate} as
\begin{equation}\label{eq:l2rate-scale}
\nm{u^{\eps}-u_0}{L^p(\Om)}\le C\eps\Lr{(\text{diam\;}\Om)^{-1}\nm{\na u_0}{L^q(\Om)}+\nm{\na^2 u_0}{L^q(\Om)}},
\end{equation}
where $C$ is independent of the diameter of $\Om$.
%Moreover, the estimate~\eqref{eq:l2rate} is scale-invariant by a scaling argument or we 
% refer to~\cite[Remark 3.4.4]{Shen:2018}.

\iffalse The following shift result will be used later on.
\begin{enumerate}
\item For $m=1$, if $\Omega$ is convex, then~\eqref{eq:regaux0} is true~\cite{Grisvard:1985}.
\item For $m\geq 2$, if $A^t=A$ and $\Omega$ is a convex polyhedron, then~\eqref{eq:regaux0} is true~\cite{MazyaRossmann:2010}.
\end{enumerate}\fi
%
\section{Error Estimates for the Periodic Media}\label{sec:error}
Before stating the main result, we make an assumption on the size of the oversampling domain $S$~\cite{ChenHou:2003}.
\noindent\vskip .3cm
{\bf Assumption A:\;}There exist constants $\gamma_1$ and $\gamma_2$ independent of $h$ such that
\[
\text{\;diam\;}S\le\gamma_1 h_{\tau}\qquad\text{and\quad}\text{dist}(\pa\tau,\pa S)\ge\gamma_2 h_{\tau}.
\]
Moreover, we always assume that $h>\eps$.
\subsection{H$^1$ error estimate}
The main result of this work is 
\begin{theorem}\label{thm:main}
Assume that $A$ is $1-$periodic and satisfies the Legendre-Hadamard condition~\eqref{eq:ellp}. For $m\geq 2$, we assume $A=A^t$. Let $\Om$ be a bounded Lipschitz domain in $\mb{R}^d$, and let $u^{\eps}$ and $u_h$ be the solutions of Problems~\eqref{eq:bvpweak} and~\eqref{eq:overMsFEM}, respectively.

For $m=1,d=2,3$ or $m\ge 2,d=2$, if $u_0\in H^2(\Om;\mb{R}^m)$, then		
\begin{equation}\label{eq:H1errb}
	\nm{u^\eps-u_h}{h,\Om}\le C\Lr{\sqrt\eps+\eps/h+h}\Lr{\nm{\na u_0}{H^1(\Om)}+\nm{f}{L^2(\Om)}},
\end{equation}
where $C$ depends on $\lam,\Lam,\Om$ and the mesh parameters $\sigma_0,\sigma_1,\gamma_1,\gamma_2$.

For $m\geq2$ and $d=3$, if $u_0\in W^{2,3}(\Om;\mb{R}^m)$, then
\begin{equation}\label{eq:H1err}
 \nm{u^\eps-u_h}{h,\Om}\le C\Lr{\sqrt\eps+\eps/h+h}\Lr{\nm{\na u_0}{W^{1,3}(\Om)}+\nm{f}{L^2(\Om)}},
\end{equation}
where $C$ depends on $\lam,\Lam, \Om$ and the mesh parameters $\sigma_0,\sigma_1,\gamma_1,\gamma_2$.
\end{theorem}

The implication of the above theorem is as follows.
\begin{enumerate}
\item The convergence rate of MsFEM proved above is the same with that in~\cite{EfendievHouWu:2000} for the scalar elliptic problem in two dimension, while we remove the superfluous technical assumptions on the coefficient $a^\eps$, the homogenized solution $u_0$ and the correctors $\chi$. 

\item The convergence rate of MsFEM is new for elliptic systems as well as problems in three dimension. 

\item We clarify the dependence of the right-hand side of the energy error estimates on $u_0$ and $f$ in the natural Sobolev norms, which together with the Aubin-Nitsche dual argument yields the convergence rate of MsFEM in L$^{d/(d-1)}-$norm. In particular, we obtain the L$^2$ error estimate for problem in $d=2$ and scalar elliptic problem in  $d=3$, cf. Theorem~\ref{thm:l2err}.%, which improves the known results in several aspects. 

\item It would be interesting to know whether Assumption A can be removed or to what degree it can be weakened. One may start with making clear how the constants $C$ in~\eqref{eq:H1errb} and~\eqref{eq:H1err} depend on $\gamma_1$ and $\gamma_2$.  Insightful discussion on this point may be found in~\cite{HenningPeterseim:2013}.
\end{enumerate}

The proof of Theorem~\ref{thm:main} is based on {\em the second lemma of Strang}~\cite{Berger:1972} because MsFEM with oversampling is a nonconforming method. 
\begin{equation}\label{eq:strang}
\nm{u^\eps-u_h}{h}\le C\Lr{\inf_{v\in V_h^0}\nm{u^\eps-v}{h}+\sup_{w\in V_h^0}
\dfrac{\abs{\dual{f}{w}-a_h(u^\eps,w)}}{\nm{w}{h}}},
\end{equation}
where $C$ depends on $\lam,\Lam,\gamma_1$ and $\gamma_2$. Therefore, the error estimate boils down to bounding the approximation error and the consistency error. To this end, we firstly define a MsFEM interpolant on each element $\tau\in\mc{T}_h$ as 
\begin{equation}\label{eq:mlinterpolant}
\wt{u}(x)|_{\tau}{:}=\sum_{i=1}^{d+1}u_0(x_i)\phi_i(x),
\end{equation}
which may be written as
\(
\wt{u}^{\beta}=\sum_{i=1}^{d+1}\sum_{k=1}^{d+1}u_0^\beta(x_i)c_{ik}^\beta\psi_k^\beta(x).
\)
It is well-defined over $S$, and
\[
\mc{L}_\eps(\wt{u})=0\quad\text{in\quad}S\qquad\text{and}\quad\wt{u}=\wt{u}_0\quad\text{on\quad}\pa S,
\]
where $\wt{u}_0^{\beta}=\sum_{i=1}^{d+1}\sum_{k=1}^{d+1}u_0^\beta(x_i)c_{ik}^\beta Q_k^\beta(x)$. 
It is clear that the homogenization limit of $\wt{u}$ is $\wt{u}_0$.
By definition, $\wt{u}_0|_{\tau}=\pi u_0$ with $\pi u_0$ the linear Lagrange interpolant of $u_0$ over $\tau$. The first order approximation of $\wt{u}$ is defined as
\[
\wt{u}_1^{\eps}{:}=\wt{u}_0+\eps(\chi\cdot\na)\wt{u}_0\quad\text{and\quad}
\wt{u}_1^{\eps}|_{\tau}=\pi u_0+\eps(\chi\cdot\na)\pi u_0.
\]

The approximation error of the MsFEM interpolant is given by
\begin{lemma}\label{lema:app}
Under the same assumptions in Theorem~\ref{thm:main}, for $m=1,d=2,3$ or $m\ge 2, d=2$, there holds
\begin{equation}\label{eq:appb}
	\nm{u^\eps-\wt{u}}{h}\le C\Lr{(\sqrt\eps+h)\nm{\na u_0}{H^1(\Om)}+\dfrac{\eps}h\nm{\na u_0}{L^2(\Om)}},
\end{equation}
where $C$ depends on $\lam,\Lam,\Om$ and the mesh parameters $\sigma_0,\sigma_1,\gamma_1,\gamma_2$.

Furthermore,  for $m\ge 2$ and $d=3$, there holds
\begin{equation}\label{eq:appa}
	\nm{u^\eps-\wt{u}}{h}\le C\Lr{(\sqrt\eps+h)\nm{\na u_0}{W^{1,3}(\Om)}+\dfrac{\eps}h\nm{\na u_0}{L^2(\Om)}},
\end{equation}
where $C$ depends on $\lam,\Lam,\Om$ and the mesh parameters $\sigma_0,\sigma_1,\gamma_1,\gamma_2$.
\end{lemma}

\begin{remark}
The interpolation estimate~\eqref{eq:appa} is new, while~\eqref{eq:appb} with $m=1$ and $d=2$ was proved in~\cite{EfendievHouWu:2000} by assuming that $\na\chi$ is bounded. The proof therein does not apply to elliptic systems because the maximum principle used in the proof may fail for elliptic
systems~\cite{MazyaRossmann:2010}. We shall use the local multiplier estimates in Lemma~\ref{lemma:Mulplier} to remove the boundedness assumption on $\na\chi$. 
\end{remark}

The next lemma concerns the estimate of the consistency error.
\begin{lemma}\label{lema:consis}
Under the same assumptions in Theorem~\ref{thm:main}, for $m=1,d=2,3$ or $m\ge2,d=2$, there holds
\begin{equation}\label{eq:consisb}
	\sup_{w\in V_h^0}\dfrac{\abs{\dual{f}{w}-a_h(u^\eps,w)}}{\nm{w}{h}}\le C\Lr{\eps+\eps/h}%\dfrac{\eps}h}
	\Lr{\nm{\na u_0}{H^1(\Om)}+\nm{f}{L^2(\Om)}}.
\end{equation}
where $C$ depends on $\lam,\Lam,\Om$ and the mesh parameters $\sigma_0,\sigma_1,\gamma_1,\gamma_2$.

For $m\geq2$ and $d=3$, there holds
\begin{equation}\label{eq:consis}
	\sup_{w\in V_h^0}\dfrac{\abs{\dual{f}{w}-a_h(u^\eps,w)}}{\nm{w}{h}}\le C\Lr{\eps+\eps/h}%\dfrac{\eps}h}
	\Lr{\nm{\na u_0}{W^{1,3}(\Om)}+\nm{f}{L^2(\Om)}},
\end{equation}
where $C$ depends on $\lam,\Lam,\Om$ and the mesh parameters $\sigma_0,\sigma_1,\gamma_1,\gamma_2$.
%the chunkiness parameter $\sigma$.
\end{lemma}
\vskip .3cm
\noindent
{\em Proof of Theorem~\ref{thm:main}\;}
Substituting Lemma~\ref{lema:app} and Lemma~\ref{lema:consis} into~\eqref{eq:strang}, we get Theorem~\ref{thm:main}.\qed
\subsubsection{Technical Results}
The main ingredients in proving Lemma~\ref{lema:app} and Lemma \ref{lema:consis} are the following local multiplier estimate, which controls the $L^2-$norm of $(\na\chi)\psi$ for certain $\psi$, and a local estimate of $\na u_1^\eps$; cf. Lemma~\ref{lemma:local1st}.
\begin{lemma}\label{lemma:Mulplier}
Let $\chi$ be defined in~\eqref{eq:corrector} and suppose that $D$ is a convex polyhedron. For any $\psi\in W^{1,d}(D;\mb{R}^m)$, there exists $C$ independent of the size of $D$ such that
\begin{equation}\label{eq:mulpliera}
\eps\nm{\na\chi\xe\psi}{L^2(D)}\le C\abs{D}^{1/2-1/d}\Lr{\nm{\psi}{L^d(D)}
+\eps\nm{\na\psi}{L^d(D)}}.
\end{equation}

If $\nm{\chi}{L^\infty}$ is bounded, then for any $\psi\in H^1(D;\mb{R}^m)$, there exists $C$ independent of the size of $D$ such that
\begin{equation}\label{eq:mulplierb}
\eps\nm{\na\chi\xe\psi}{L^2(D)}\le C(1+\nm{\chi}{L^\infty})\Lr{\nm{\psi}{L^2(D)}
		+\eps\nm{\na\psi}{L^2(D)}}.
\end{equation}
\end{lemma}

The proof depends on the following multiplier estimates proved in~\cite[Lemma 3.2.8]{Shen:2018}: For any $\psi\in W^{1,d}(\Om;\mb{R}^m)$,
\begin{equation}\label{eq:multipliera}
\eps\nm{\na\chi\xe\psi}{L^2(\Om)}\le C\Lr{\nm{\psi}{L^d(\Om)}+\eps\nm{\na\psi}{L^d(\Om)}},
\end{equation}
and for any $\psi\in H^1(\Om;\mb{R}^m)$,
\begin{equation}\label{eq:multiplierb}
\eps\nm{\na\chi\xe\psi}{L^2(\Om)}\le C(1+\nm{\chi}{L^\infty})\Lr{\nm{\psi}{L^2(\Om)}+\eps\nm{\na\psi}{L^2(\Om)}},
\end{equation}
where $C$ depends on $\lam,\Lam$ and $\Om$. These multiplier estimates are crucial to prove the error 
bounds~\eqref{eq:1stratea} and~\eqref{eq:1strateb}. These estimates have been refined in Lemma~\ref{lemma:Mulplier} by tracing the dependence of the constant on the size of the domain.

\begin{proof}
Denote $L=\text{diam\;}D$, and we apply the scaling $x'=x/L$ to $D$ so that the rescaled element $\wh{D}$ has diameter $1$. Note that
\[
x/\eps=x'/\eps'\qquad\text{with}\quad\eps'=\eps/L.
\]
Hence $\eps\na\chi(x/\eps)=\eps'\;\na_{x'}\chi(x'/\eps')$ and $\psi(x)=\psi(Lx')=\wh{\psi}(x')$. Applying~\eqref{eq:multipliera} to $\wh{D}$, we obtain that there exists $C$ depends only on $\wh{D}$ such that
\begin{align*}
\eps\nm{\na\chi\xe\psi}{L^2(D)}&\le\Lr{\abs{D}/\abs{\wh{D}}}^{1/2}\eps'\nm{\na_{x'}\chi(x'/\eps')\wh{\psi}}{L^2(\wh{D})}\\
&\le C\abs{D}^{1/2}\Lr{\nm{\wh{\psi}}{L^d(\wh{D})}+\eps'\nm{\na_{x'}\wh{\psi}}{L^d(\wh{D})}}\\
&\le C\abs{D}^{1/2-1/d}\Lr{\nm{\psi}{L^d(D)}+\eps\nm{\na\psi}{L^d(D)}}.
\end{align*}
This yields~\eqref{eq:mulpliera}.
	
Replacing~\eqref{eq:multipliera} by~\eqref{eq:multiplierb} and proceeding along the same line that leads to~\eqref{eq:mulpliera}, we obtain~\eqref{eq:mulplierb}.
\end{proof}

Another ingredient of the error estimate is the quantitative estimates for the MsFEM functions in $V_h$, which have been used in all the previous study. For any $w\in V_h$, we may write, on each element $\tau\in\mc{T}_h$, 
\[
w^{\beta}(x)|_{\tau}{:}=\sum_{i=1}^{d+1}w_i\phi_i(x)=\sum_{i=1}^{d+1}\sum_{k=1}^{d+1}w_i^\beta c_{ik}^\beta\psi_k^\beta(x)
\]
for certain coefficients $w_i\in\mb{R}^m$. It is well-defined over $S$, and
\[
\mc{L}_\eps(w)=0\quad\text{in\quad}S\qquad\text{and}\quad w=w_0\quad\text{on\quad}\pa S,
\]
where $w_0^{\beta}=\sum_{i=1}^{d+1}\sum_{k=1}^{d+1}w_i^\beta c_{ik}^\beta Q_k^\beta(x)$. It is clear that the homogenization limit of $w$ is $w_0$, and there exists $C$ depending on $\lam,\Lam,\gamma_1$ and $\gamma_2$, but independent of $\eps$ and $h_{\tau}$, such that
\begin{equation}\label{eq:mslinear}
	\nm{\na w_0}{L^2(\tau)}\le C\nm{\na w}{L^2(\tau)}\qquad\text{for all\quad}\tau\in\mc{T}_h.
\end{equation}
This inequality was proved in~\cite[Appendix B]{EfendievHouWu:2000}. The first order approximation of $w$ is defined by
\(
w_1^{\eps}{:}=w_0+\eps(\chi\cdot\na)w_0.
\)
\begin{lemma}\label{lema:localhomo}
Suppose that \textbf{Assumption A} is true and $A=A^t$ for $m\ge 2$. For $w\in V_h$, there exists $C$ such that
%independent of $\eps$ and $h_{\tau}$ such that
\begin{equation}\label{eq:l2local}
\nm{w-w_0}{L^2(S)}\le C\eps\nm{\na w_0}{L^2(S)},
\end{equation}
and
\begin{equation}\label{eq:h1local}
\nm{\na(w-w_1^\eps)}{L^2(\tau)}\le C\dfrac{\eps}{h_{\tau}}\nm{\na w_0}{L^2(S)}.
\end{equation}
\end{lemma}

\begin{proof}
Applying Theorem~\ref{thm:scale-invariant} to $w$, using~\eqref{eq:l2rate-scale} and the fact that $w_0$ is linear over $S$, we obtain
\begin{align*}
\nm{w-w_0}{L^2(S)}&\le\abs{S}^{1/2-1/p}\nm{w-w_0}{L^p(S)}\\
&\le C\eps\abs{S}^{1/2-1/p}\Lr{(\text{diam\,}S)^{-1}\nm{\na w_0}{L^q(S)}+\nm{\na^2 w_0}{L^q(S)}}\\
&=C\dfrac{\eps}{\text{diam\,}S}\abs{S}^{1/2-1/p+1/q}\abs{\na w_0}\\
&\le C\eps\nm{\na w_0}{L^2(S)},
\end{align*}
where we have used $1/q-1/p=1/d$ in the last step. This gives~\eqref{eq:l2local}.

Note that
\[
a_S(w-w_1^\eps,v)=0\qquad\text{for all\quad}v\in H_0^1(S;\mb{R}^m).
\]
By the Caccioppoli inequality~\cite[Corollary 1.37]{HanLin:1997} and \textbf{Assumption A}, there exists $C$ that depends on $\lam,\Lam,\gamma_1$ and $\gamma_2$ such that
\begin{equation}\label{eq:coccipolli}
\nm{\na(w-w_1^\eps)}{L^2(\tau)}\le\dfrac{C}{h_{\tau}}\nm{w-w_1^\eps}{L^2(S)}.
\end{equation}%

Using the fact that $\na w_0$ is a piecewise constant matrix and~\eqref{eq:corr-est} with $p=2$, we obtain
\begin{align*}
\nm{w_1^\eps-w_0}{L^2(S)}&=\eps\nm{\chi\xe\na w_0}{L^2(S)}
=\eps\nm{\chi\xe}{L^2(S)}\abs{\na w_0}\\
&\le C\eps\abs{S}^{1/2}\nm{\chi}{L^2(Y)}\abs{\na w_0}=C\eps\nm{\chi}{L^2(Y)}\nm{\na w_0}{L^2(S)},
\end{align*}
which together with~\eqref{eq:l2local} and the triangle inequality gives
\[
\nm{w-w_1^\eps}{L^2(S)}\le\nm{w-w_0}{L^2(S)}+\nm{w_1^\eps-w_0}{L^2(S)}
\le C\eps\nm{\na w_0}{L^2(S)}.
\]
Substituting the above inequality into~\eqref{eq:coccipolli}, we obtain~\eqref{eq:h1local}.
\end{proof}

Another useful tool is the following inequality for a tubular domain defined below. Let $\tau\in\mc{T}_h$, for any $\del>0$, we define 
\[
\tau_{\del}{:}=\set{x\in\tau}{\text{dist}(x,\pa\tau)\le\del}.
\]
\begin{lemma}\label{lemma:retrace}
Let $1\le p<\infty$, for any $v\in W^{1,p}(\tau)$, there exists $C$ depending on $p,d$ and $\sigma_0$ such that 
\begin{equation}\label{eq:trace1}
\nm{v}{L^p(\tau_\del)}\le C(\del/h_\tau)^{1/p}\nm{v}{W^{1,p}(\tau)}.
\end{equation}
\end{lemma}

This inequality has appeared in many occurrences, and we give a proof for the readers' convenience.
\begin{proof}
For any $0<s<\del$, we let $\tau_s^c=\tau\backslash\tau_s$. It is clear that $\tau_s^c$ is also a simplex. For any face $f$ of $\tau_s^c$, we define a vector
\[
m(x)=\dfrac{\abs{f}}{d\abs{\tau_s^c}}(x-a_f),
\]
where $a_f$ is the vertex opposite to $f$. A direct calculation gives that $m(x)\cdot n_f=1$ for any $x\in f$, while $m(x)\cdot n_g$ vanishes on the remaining faces of $\tau_s^c$, where $n_g$ is the outward normal of the face $g$ so that $x\in g$. Using the divergence theorem, we obtain
\begin{align*}
\int_f\abs{v(x)}^p\mathrm{d}\sigma(x)&=\int_f\abs{v(x)}^pm(x)\cdot n_f\mathrm{d}\sigma(x)=\int_{\tau_s^c}\divop\left(\abs{v(x)}^pm(x)\right)\dx\\
&=\int_{\tau_s^c}\Lr{\Lr{m(x)\cdot\na}\abs{v(x)}^p+\abs{v(x)}^p\divop m(x)}\dx.
\end{align*}
A direct calculation gives 
\[
\max_{x\in\tau_s^c}\abs{m(x)}\le\sigma_0\qquad\divop m(x)=\dfrac{\abs{f}}{\abs{\tau_s^c}}\le\dfrac{d\sigma_0}{h_{\tau}}.
\]
A combination of the above two inequalities leads to
\begin{align*}
\int_f\abs{v(x)}^p\mathrm{d}\sigma(x)&\le\sigma_0\Lr{\dfrac{d}{h_{\tau}}\int_{\tau_s^c}\abs{v(x)}^p\dx+p\int_{\tau_s^c}\abs{v(x)}^{p-1}\abs{\na v(x)}\dx}\\
&\le\dfrac{\sigma_0}{h_{\tau}}\Lr{d\int_{\tau}\abs{v(x)}^p\dx+ph_{\tau}\int_{\tau}\abs{v(x)}^{p-1}\abs{\na v(x)}\dx}.
%&\le\sigma_0\dfrac1{h_{\tau}}\int_{\tau}\abs{v(x)}^p\dx+p\sigma_0\Lr{\Lr{\int_{\tau}\abs{v(x)}^p\dx}^{p-1}\int_{\tau}
%\abs{\na v(x)}^p\dx}.
\end{align*}

Summing up all faces $f\in\pa\tau_s^c$, we obtain
\[
\int_{\pa\tau_s^c}\abs{v(x)}^p\mathrm{d}\sigma(x)\le\dfrac{(d+1)\sigma_0}{h_{\tau}}\Lr{d
\int_{\tau}\abs{v(x)}^p\dx+ph_{\tau}\int_{\tau}\abs{v(x)}^{p-1}\abs{\na v(x)}\dx}.
\]
Integration with respect to $s$ from $0$ to $\del$, we obtain
\[
\int_{\tau_{\del}}\abs{v(x)}^p\mathrm{d}\sigma(x)\le\dfrac{(d+1)\sigma_0\del}{h_{\tau}}\Lr{d\int_{\tau}\abs{v(x)}^p\dx+ph_{\tau}\int_{\tau}\abs{v(x)}^{p-1}\abs{\na v(x)}\dx}.
\]
Using H\"older's inequality, we obtain
\[
\nm{v}{L^p(\tau_{\del})}\le\Lr{\del/h_{\tau}}^{1/p}((d+1)\sigma_0)^{1/p}\Lr{d^{1/p}\nm{v}{L^p(\tau)}+(ph_{\tau})^{1/p}
\nm{v}{L^p(\tau)}^{1-1/p}\nm{\na v}{L^p(\tau)}^{1/p}}.
\]
This gives~\eqref{eq:trace1} for $p>1$.

The proof for $p=1$ is the same, we omit the details.
\end{proof}

To bound the consistency error, we need a local estimate of $\na u_1^\eps$, which helps us to remove the  extra smoothness assumption on $\chi$.
\begin{lemma}\label{lemma:local1st}
%Let $\tau\in\mc{T}_h$ be an element and $\tau_\eps$ be a strip that is defined by $\tau_\eps{:}=\set{x\in\tau}{\emph{dist}(x,\pa\tau)\le 2\eps}$. 
There exists $C$ independent of $\eps,\del$ and $h_{\tau}$ such that
\begin{equation}\label{eq:u1trunc}
\nm{\na u_1^\eps}{L^2(\tau_\del)}\le C\Lr{\eps+\sqrt{\del/h_{\tau}}}\abs{\tau}^{1/2-1/d}\nm{\na u_0}{W^{1,d}(\tau)}.
\end{equation}

If $\chi$ is bounded, then
\begin{equation}\label{eq:u1truncb}
\nm{\na u_1^\eps}{L^2(\tau_\del)}\le C\Lr{\eps+\sqrt{\del/h_{\tau}}}\Lr{1+\nm{\chi}{L^\infty(Y)}}\nm{\na u_0}{H^1(\tau)}.
\end{equation}
\end{lemma}

\begin{proof}
Since $\tau$ is a simplex, we may decompose $\tau_\del$ into  $d+1$ disjoint convex domains $\{\tau_\del^i\}_{i=1}^{d+1}$. Over each $\tau_\del^i$, using the local multiplier estimate~\eqref{eq:mulpliera}, we obtain
\[
\eps\nm{\na\chi\xe\na u_0}{L^2(\tau_\del^i)}\le C\abs{\tau_\del^i}^{1/2-1/d}\Lr{\nm{\na u_0}{L^d(\tau_\del^i)}+\eps\nm{\na^2 u_0}{L^d(\tau_\del^i)}}.
\]
Summing up the above estimate for $i=1,\ldots,d+1$ and using the scaled trace inequality ~\eqref{eq:trace1} with $p=d$, we obtain
\begin{align*}
\eps\nm{\na\chi\xe\na u_0}{L^2(\tau_\del)}&\le C\abs{\tau_\del}^{1/2-1/d}\Lr{\nm{\na u_0}{L^d(\tau_\del)}+\eps\nm{\na^2 u_0}{L^d(\tau_\del)}}\\
&\le C \abs{\tau_\del}^{1/2-1/d}(\del/h_\tau)^{1/d}\nm{\na u_0}{W^{1,d}(\tau)}\\
&\qquad+C\eps\abs{\tau}^{1/2-1/d}\nm{\na^2 u_0}{L^d(\tau)}\\
&\le C(\eps+\sqrt{\del/h_\tau})\abs{\tau}^{1/2-1/d}\nm{\na u_0}{W^{1,d}(\tau)}. 
\end{align*}

Invoking the scaled trace inequality~\eqref{eq:trace1} with $p=2$ and using H\"older's inequality, we obtain
\[
\nm{\na u_0}{L^2(\tau_\del)}\le C\sqrt{\del/h_{\tau}}\nm{\na u_0}{H^1(\tau)}
\le C\sqrt{\del/h_{\tau}}\abs{\tau}^{1/2-1/d}\nm{\na u_0}{W^{1,d}(\tau)}.
\]

Using H\"older's inequality with $1/q=1/2-1/d$ and~\eqref{eq:corr-est} with $p=q$, we obtain
\begin{align*}
	\eps\nm{\chi\xe\na^2 u_0}{L^2(\tau_\del)}&\le\eps\nm{\chi\xe\na^2 u_0}{L^2(\tau)}
	\le\eps\nm{\chi\xe}{L^q(\tau)}\nm{\na^2 u_0}{L^d(\tau)}\\
	&\le C\eps\abs{\tau}^{1/2-1/d}\nm{\chi}{L^q(Y)}\nm{\na^2 u_0}{L^d(\tau)}.
\end{align*}%where $s=2d/(d-2)$. 

\iffalse Applying the above three inequalities to
\[
\na u_1^\eps=\na u_0+\eps\na\chi\xe\na u_0+\eps\chi\xe\na^2 u_0
\]\fi
A combination of the above three inequalities leads to~\eqref{eq:u1trunc}.

If $\chi$ is bounded, then we sum up the local multiplier estimate~\eqref{eq:mulplierb} over $\tau_\del^i$ for $i=1,\ldots,d+1$ and obtain
\[
\eps\nm{\na\chi\xe\na u_0}{L^2(\tau_\del)}\le C(1+\nm{\chi}{L^\infty(Y)})\Lr{\nm{\na u_0}{L^2(\tau_\del)}+\eps\nm{\na^2 u_0}{L^2(\tau_\del)}}.
\]
Invoking the scaled trace inequality~\eqref{eq:trace1} again, we obtain
\begin{align*}
\nm{\na u_1^\eps}{L^2(\tau_\del)}&\le\nm{\na u_0}{L^2(\tau_\del)}+{ \eps}\nm{\na\chi\xe\na u_0}{L^2(\tau_\del)}
+\eps\nm{\chi\na^2 u_0}{L^2(\tau_\del)}\\
&\le C(1+\nm{\chi}{L^\infty(Y)})\Lr{\nm{\na u_0}{L^2(\tau_\del)}
+\eps\nm{\na^2 u_0}{L^2(\tau)}}\\
%&\le C\sqrt{\eps/h_{\tau}}\nm{\na^2 u_0}{L^2(\tau)}+C\eps(1+\nm{\chi}{L^\infty(Y)})\nm{\na^2u_0}{L^2(\tau)}\\
&\le C\Lr{\eps+\sqrt{\del/h_{\tau}}}(1+\nm{\chi}{L^\infty(Y)})\nm{\na^2u_0}{L^2(\tau)}.
\end{align*}
This gives~\eqref{eq:u1truncb} and finishes the proof.
\end{proof}
\subsubsection{Proof of Lemma~\ref{lema:app} and Lemma~\ref{lema:consis}}
\noindent
\vskip .3cm
%for the approximation error and consistency error}
{\em Proof for Lemma~\ref{lema:app}\;\;}Using the triangle inequality, we have
\begin{equation}\label{eq:dec}
\begin{aligned}
\nm{u^\eps-\wt{u}}{h}&\le\nm{u^\eps-u_1^{\eps}}{h}+\nm{\wt{u}-\wt{u}_1^\eps}{h}+\nm{u_1^{\eps}-\wt{u}_1^\eps}{h}\\
&=\nm{\na(u^\eps-u_1^\eps)}{L^2(\Om)}
+\nm{\wt{u}-\wt{u}_1^\eps}{h}+\nm{u_1^{\eps}-\wt{u}_1^\eps}{h}.
\end{aligned}
\end{equation}

Applying Lemma~\ref{lema:localhomo} to $\wt{u}$, using~\eqref{eq:h1local} and \textbf{Assumption A}, we obtain
%\begin{equation}\label{eq:1stapploc}
\begin{align*}
\nm{\na(\wt{u}-\wt{u}_1^\eps)}{L^2(\tau)}&\le C\dfrac{\eps}{h_{\tau}}\nm{\na\wt{u}_0}{L^2(S)}=C\dfrac{\eps}{h_{\tau}}\abs{S}^{1/2}\abs{\na\wt{u}_0}\\
&=C\dfrac{\eps}{h_{\tau}}\abs{S}^{1/2}\abs{\na\pi u_0}=C\dfrac{\eps}{h_{\tau}}\dfrac{\abs{S}^{1/2}}{\abs{\tau}^{1/2}}\nm{\na\pi u_0}{L^2(\tau)}\\
&\le C\dfrac{\eps}{h_{\tau}}\nm{\na\pi u_0}{L^2(\tau)}.
\end{align*}
%\end{equation}

Summing up all $\tau\in\mc{T}_h$, using the shape-regular and inverse assumption of $\mc{T}_h$, we obtain 
%that there exists $C$ depending on $\sigma_0$ and $\sigma_1$ such that
\begin{align}\label{eq:1stele-err}
\nm{\wt{u}-\wt{u}_1^\eps}{h}&\le C\dfrac{\eps}{h}\nm{\na\pi u_0}{L^2(\Om)}\le C\dfrac{\eps}{h}\Lr{\nm{\na(u_0-\pi u_0)}{L^2(\Om)}
+\nm{\na u_0}{L^2(\Om)}}\nn\\
&\le C\Lr{\eps\nm{\na^2 u_0}{L^2(\Om)}+\dfrac{\eps}{h}\nm{\na u_0}{L^2(\Om)}}.
\end{align}

On each element $\tau,u_1^\eps-\wt{u}_1^{\eps}=u_0-\pi u_0+\eps\chi\xe\na(u_0-\pi u_0)$ and
\[
\na(u_1^\eps-\wt{u}_1^{\eps})=\na(u_0-\pi u_0)+\eps\na\chi\xe\na(u_0-\pi u_0)+\eps\chi\xe\na^2u_0.
\]
For $m=1,d=2,3$ or $m\ge 2,d=2$, $\chi$ is bounded by~\eqref{eq:maxbd}, using the local multiplier inequality~\eqref{eq:mulplierb}, we obtain
\begin{align*}
\eps\nm{\na\chi\xe\na(u_0-\pi u_0)}{L^2(\tau)}&\le C\Lr{\nm{\na(u_0-\pi u_0)}{L^2(\tau)}
+\eps\nm{\na^2u_0}{L^2(\tau)}}\\
&\le C(\eps+h_{\tau})\nm{\na^2 u_0}{L^2(\tau)}.
\end{align*}
It follows from the above two equations that
\begin{align*}
\nm{\na(u_1^\eps-\wt{u}_1^{\eps})}{L^2(\tau)}&\le\nm{\na(u_0-\pi u_0)}{L^2(\tau)}
+\eps\nm{\na\chi\xe\na(u_0-\pi u_0)}{L^2(\tau)}\\
&\quad+\eps\nm{\chi\xe\na^2 u_0}{L^2(\tau)}\\
&\le C\Lr{1+\nm{\chi}{L^\infty(Y)}}(\eps+h_{\tau})\nm{\na^2 u_0}{L^2(\tau)}.
\end{align*}
%Here $\na_y\chi_\eps$ is understood as $\na_y\chi(y)|_{y=x/\eps}$.
Summing up all $\tau\in\mc{T}_h$, and using~\eqref{eq:maxbd} again, we get
\begin{equation}\label{eq:1stcompare}
\nm{u_1^\eps-\wt{u}_1^{\eps}}{h}\le C(\eps+h)\nm{\na^2 u_0}{L^2(\Om)}.
\end{equation}

Substituting the above inequality,~\eqref{eq:1strateb} and~\eqref{eq:1stele-err} into~\eqref{eq:dec}, we obtain~\eqref{eq:appb}.

For $m\ge 2$ and $d=3$, by~\eqref{eq:corr-reg}, we have $\chi\in L^6(Y)$. Using the local multiplier estimate~\eqref{eq:mulpliera} and the standard interpolation estimate for $\pi u_0$, we obtain
\begin{align*}
\eps\nm{\na\chi\xe\na(u_0-\pi u_0)}{L^2(\tau)}&\le C\abs{\tau}^{1/6}
\Lr{\nm{\na(u_0-\pi u_0)}{L^3(\tau)}
+\eps\nm{\na^2u_0}{L^3(\tau)}}\\
&\le C(\eps+h_{\tau})\abs{\tau}^{1/6}\nm{\na^2 u_0}{L^3(\tau)}.
\end{align*}
Using H\"older's inequality, the inequality~\eqref{eq:corr-est} with $p=6,D=\tau$ and~\eqref{eq:corr-reg}, we obtain
\[
\eps\nm{\chi\xe\na^2 u_0}{L^2(\tau)}\le\eps\nm{\chi\xe}{L^6(\tau)}\nm{\na^2 u_0}{L^3(\tau)}\le C\eps\abs{\tau}^{1/6}\nm{\na^2 u_0}{L^3(\tau)}.
\]
Proceeding along the same line that leads to~\eqref{eq:1stcompare}, we obtain
\[
\nm{\na(u_1^\eps-\wt{u}_1^{\eps})}{L^2(\tau)}\le C(\eps+h_{\tau})\abs{\tau}^{1/6}\nm{\na^2u_0}{L^3(\tau)}.
\]
Summing up all $\tau\in\mc{T}_h$ and using H\"older's inequality, we get
\[
\nm{u_1^\eps-\wt{u}_1^{\eps}}{h}\le C(\eps+h)\nm{\na^2 u_0}{L^3(\Om)}.
\]

Substituting the above inequality,~\eqref{eq:1stratea} and~\eqref{eq:1stele-err} into~\eqref{eq:dec}, we obtain~\eqref{eq:appa}.
\qed
%\subsubsection{Consistency error}
\vskip .3cm\noindent{\em Proof for Lemma~\ref{lema:consis}\;}
For $w\in V_h^0$, over each oversampling domain $S$, let $w_0$ be its homogenized part over $S$. %It is clear that $w_0=\pi w_0$ over each element $\tau$. 
By $w_0\in H_0^1(\Om;\R^m)$, there holds
\[
a_h(u^\eps,w_0)=\dual{f}{w_0}.
\]
Therefore, we write the consistency error functional as
\begin{align*}
\dual{f}{w}-a_h(u^\eps,w)&=\dual{f}{w-w_0}-a_h(u^\eps,w-w_0)\\
&=\dual{f}{w-w_0}-a_h(u^\eps,w-w_1^\eps)-a_h(u^\eps,w_1^\eps-w_0).
\end{align*}
%where $w_1^\eps=w_0+\eps\chi\xe\na w_0$.

Using Lemma~\ref{lema:localhomo},~\eqref{eq:l2local},~\eqref{eq:mslinear} and \textbf{Assumption} A, we obtain
\begin{align*}
\nm{w-w_0}{L^2(\tau)}&\le\nm{w-w_0}{L^2(S)}\le C\eps\nm{\na w_0}{L^2(S)}\\
%&\le C\eps\nm{\na\pi_{\text{ms}} w}{L^2(\tau)}
&\le C\eps\nm{\na w_0}{L^2(\tau)}\le C\eps\nm{\na w}{L^2(\tau)},
\end{align*}
which immediately implies
\begin{equation}\label{eq:sourceconsis}
\abs{\dual{f}{w-w_0}}\le C\eps\nm{f}{L^2(\Om)}\nm{w}{h}.
\end{equation}

Using~\eqref{eq:h1local},~\eqref{eq:mslinear} again, and the inverse assumption of $\mc{T}_h$, we obtain
\begin{align*}
\abs{a_h(u^\eps,w-w_1^\eps)}&\le\Lam\sum_{\tau\in\mc{T}_h}\nm{\na u^\eps}{L^2(\tau)}\nm{\na(w-w_1^\eps)}{L^2(\tau)}\\
&\le C\sum_{\tau\in\mc{T}_h}\dfrac{\eps}{h_{\tau}}\nm{\na u^\eps}{L^2(\tau)}\nm{\na w_0}{L^2(\tau)}\\
&\le C\dfrac{\eps}{h}\sum_{\tau\in\mc{T}_h}\nm{\na u^\eps}{L^2(\tau)}\nm{\na w}{L^2(\tau)}\\
&\le C\dfrac{\eps}{h}\nm{\na u^\eps}{L^2(\Om)}\nm{w}{h}.
\end{align*}
Combining the above two estimates, we obtain
\begin{equation}\label{eq:consis1}
\abs{\dual{f}{w-w_0}-a_h(u^\eps,w-w_1^\eps)}\le C\Lr{\eps+\eps/h}\nm{f}{L^2(\Om)}\nm{w}{h},
\end{equation}
where we have used the a-priori estimate $\nm{\na u^\eps}{L^2(\Om)}\le C\nm{f}{L^2(\Om)}$.
\iffalse
Over each element $\tau$, we have $w_1^\eps-\wt{w}_0=\eps\chi\xe\na\pi w_0$. Therefore,
\begin{align*}
\nm{w_1^\eps-\wt{w}_0}{L^2(\tau)}&=\eps\nm{\chi\xe}{L^2(\tau)}\abs{\na\pi w_0}\le C\eps\abs{\tau}^{1/2}
\nm{\chi}{L^2(Y)}\abs{\na\pi w_0}\\
&=C\eps\nm{\chi}{L^2(Y)}\nm{\na\pi w_0}{L^2(\tau)}\\
&\le C\eps\nm{\na w}{L^2(\tau)},
\end{align*}
which immediately implies
\[
\abs{\dual{f}{w_1^\eps-\wt{w}_0}}\le\sum_{\tau\in\mc{T}_h}\nm{f}{L^2(\tau)}\nm{w_1^\eps-\wt{w}_0}{L^2(\tau)}
\le C\eps\nm{f}{L^2(\Om)}\nm{w}{h}.
\]\fi

It remains to bound $a_h(u^\eps,w_1^\eps-w_0)$. On each element $\tau$, we  introduce a cut-off function $\rho_\eps\in C_0^\infty(\tau)$ such that $0\le\rho_\eps\le 1$ and $\abs{\na\rho_\eps}\le C/\eps$, moreover,
\[
\rho_\eps=\left\{\begin{aligned}
1\quad&\text{dist}(x,\pa\tau)\ge 2\eps,\\
0\quad&\text{dist}(x,\pa\tau)\le\eps.
\end{aligned}\right.
\]
Denote $\wh{w}^\eps=(w_1^\eps-w_0)(1-\rho_\eps)$, which is the oscillatory part of $w_1^\eps$ 
supported inside the strip $\tau_{2\eps}$. We write
\begin{align*}
a_\tau(u^\eps,w_1^\eps-w_0)=&a_\tau(u^\eps,(w_1^\eps-w_0)\rho_\eps)
+a_\tau(u^\eps,\wh{w}^\eps)\\
=&\dual{f}{(w_1^\eps-w_0)\rho_\eps}_{\tau}+a_\tau(u^\eps,\wh{w}^\eps).
%&+a_\tau(u_1^\eps,(w_1^\eps-\pi w_0)(1-\rho_\eps)).
\end{align*}%Noting that $w_1^\eps-\pi w_0=\eps\chi\xe\na\pi w_0$, 
Using~\eqref{eq:corr-est} with $p=2$, we obtain
\begin{equation}\label{eq:consis21}
\begin{aligned}
\abs{\dual{f}{(w_1^\eps-w_0)\rho_\eps}_{\tau}}&\le\eps\nm{f}{L^2(\tau)}\nm{\chi\xe}{L^2(\tau)}\abs{\na w_0}\\
&\le C\eps\abs{\tau}^{1/2}\nm{f}{L^2(\tau)}\nm{\chi}{L^2(Y)}\abs{\na w_0}\\
&= C\eps\nm{f}{L^2(\tau)}\nm{\chi}{L^2(Y)}\nm{\na w_0}{L^2(\tau)}.\\
\end{aligned}
\end{equation}

A direct calculation gives\footnote{We may also refer to~\cite[Lemma 3.1]{DuMing:2010} for a proof of~\eqref{eq:cutoffest1}.}
\begin{equation}\label{eq:cutoffest1}
\nm{\na\wh{w}^\eps}{L^2(\tau_{2\eps})}\le C\sqrt{\eps/h_{\tau}}\nm{\na w_0}{L^2(\tau)},
\end{equation}
which together with the local estimate~\eqref{eq:u1trunc} implies that, for $m\geq2$ and $d=3$, there holds
\[
\begin{aligned}
&\abs{a_\tau(u^\eps,\wh{w}^\eps)}
\le\abs{a_\tau(u_1^\eps,\wh{w}^\eps)}
+\abs{a_\tau(u^\eps-u_1^\eps,\wh{w}^\eps)}\\
&\le C\Lr{\Lr{\eps+\dfrac{\eps}{h_{\tau}}}\abs{\tau}^{1/6}\nm{\na u_0}{W^{1,3}(\tau)}
+\sqrt{\dfrac{\eps}{h_{\tau}}}\nm{\na (u^\eps-u_1^\eps)}{L^2(\tau)}}\nm{\na w_0}{L^2(\tau)}.
\end{aligned}
\]
This estimate together with~\eqref{eq:consis21} implies
\begin{align*}
\abs{a_{\tau}(u^\eps,w_1^\eps-w_0)}&\le C\Bigl(\Lr{\eps+\dfrac{\eps}{h_{\tau}}}\abs{\tau}^{1/6}\nm{\na u_0}{W^{1,3}(\tau)}
+\sqrt{\dfrac{\eps}{h_{\tau}}}\nm{\na (u^\eps-u_1^\eps)}{L^2(\tau)}\\
&\qquad\quad+\eps\nm{f}{L^2(\tau)}\Bigr)\nm{\na w_0}{L^2(\tau)}.
\end{align*}
Summing up the above estimates for all $\tau\in\mc{T}_h$, using~\eqref{eq:mslinear},~\eqref{eq:1strateb}, the inverse assumption of $\mc{T}_h$ and H\"older's inequality, we obtain
\begin{align*}
\abs{a_h(u^\eps,w_1^\eps-w_0)}&\le C\Bigl(\Lr{\eps+\dfrac{\eps}{h}}\nm{\na u_0}{W^{1,3}(\Om)}
+\sqrt{\dfrac{\eps}h}\nm{\na(u^\eps-u^\eps_1)}{L^2(\Om)}\\
&\qquad\quad+\eps\nm{f}{L^2(\Om)}\Bigr)\nm{w}{h}\\
&\le C\Lr{\eps+\dfrac{\eps}{h}}\Lr{\nm{\na u_0}{W^{1,3}(\Om)}+\nm{f}{L^2(\Om)}}\nm{w}{h}.
\end{align*}%
This inequality together with~\eqref{eq:consis1} implies \eqref{eq:consis}.

For $m=1,d=2,3$ or $m\ge 2,d=2$, $\chi$ is bounded. Replacing~\eqref{eq:u1trunc} by~\eqref{eq:u1truncb} and proceeding along the same line that leads to \eqref{eq:consis}, we obtain~\eqref{eq:consisb}.
\qed
\subsection{L$^{d/(d-1)}$ error estimate}
We exploit the Aubin-Nitsche trick to obtain the error estimate of MsFEM in $L^{d/(d-1)}-$norm with $d=2,3$.
\begin{theorem}\label{thm:l2err}
Under the same assumption of Theorem~\ref{thm:main}, and suppose that $\varphi\in H_0^1(\Om;\mb{R}^m)$ satisfying
\[
\int_{\Om}\na\varphi\cdot\wh{A}\na\psi\dx=\dual{F}{\psi}\qquad\text{for all\quad}\psi\in H_0^1(\Om;\mb{R}^m).
\]
For $m=1,d=2,3$ or $m\ge 2,d=2$, if the shift estimate
	\begin{equation}\label{eq:regaux1}
		\nm{\varphi}{H^2(\Om)}\le C\nm{F}{L^2(\Om)}
	\end{equation}
holds true, then for $m=1,d=2,3$, there holds
	\begin{equation}\label{eq:l2errold}
		\nm{u-u_h}{L^2(\Om)}\le C(\eps+h^2+\eps/h)\Lr{\nm{\na u_0}{H^1(\Om)}+\nm{f}{L^2(\Om)}}.
	\end{equation}
For $m\ge 2,d=2$, there holds
\begin{equation}\label{eq:l2err1}
		\nm{u-u_h}{L^2(\Om)}\le C(\eps+h^2+\eps/h)\nm{f}{L^2(\Om)}.
\end{equation}
	
For $m\ge 2$ and $d=3$, if the shift estimate
\begin{equation}\label{eq:regaux}
		\nm{\varphi}{W^{2,3}(\Om)}\le C\nm{F}{L^3(\Om)}
\end{equation}
holds true, then  
\begin{equation}\label{eq:l2err}
		\nm{u-u_h}{L^{3/2}(\Om)}\le C(\eps+h^2+\eps/h)\nm{f}{L^3(\Om)}.
	\end{equation}
\end{theorem}

Except the resonance error $\eps/h$, the other two items in the above error estimates are {\em optimal}. For scalar elliptic equation and elliptic systems in two dimension, we obtain the L$^2$ error estimate.
\begin{proof}
For any $g\in L^2(\Om;\mb{R}^m)$, we find $v^\eps\in H_0^1(\Om;\mb{R}^m)$ such that
\begin{equation}\label{eq:aux}
\int_\Om\na w\cdot(A\xe)^t\na v^\eps\dx=\int_{\Om}g\cdot w\dx\qquad\text{for all\quad}w\in H_0^1(\Om;\mb{R}^m).
\end{equation}
Let $v_h$ be the MsFEM approximation of $v^\eps$ defined by
\begin{equation}\label{eq:auxmsfem}
a_h(w,v_h)=\int_{\Om}g\cdot w\dx\qquad\text{for all\quad} w\in V_h^0.
\end{equation}
	
It follows from~\eqref{eq:aux} and ~\eqref{eq:auxmsfem} that
	%\begin{equation}\label{eq:dualpro}
\[
\begin{aligned}
\int_{\Om}g\cdot(u^\eps-u_h)\dx&=a(u^\eps,v^\eps)-a_h(u_h,v_h)\\
&=a_h(u^\eps-u_h,v^\eps-v_h)+a_h(u^\eps-u_h,v_h)+a_h(u_h,v^\eps-v_h)\\
&=a_h(u^\eps-u_h,v^\eps-v_h)\\
&\quad+\bigl[a_h(u^\eps,v_h)-\dual{f}{v_h}+a_h(u_h,v^\eps)-\dual{g}{u_h}\bigr].
\end{aligned}
\]
For $m=1,d=2,3$ or $m\ge 2,d=2$, using the energy error estimate~\eqref{eq:H1errb} and the regularity assumption~\eqref{eq:regaux1}, we obtain
%\begin{equation}\label{eq:dual1st}
\begin{align*}
\abs{a_h(u^\eps-u_h,v^\eps-v_h)}&\le\Lam\nm{u^\eps-u_h}{h}\nm{v^\eps-v_h}{h}\\
&\le C(\eps+h^2+\eps^2/h^2)\Lr{\nm{\na u_0}{H^1(\Om)}+\nm{f}{L^2(\Om)}}\nm{g}{L^2(\Om)}.
\end{align*}
%\end{equation}
	
Using~\eqref{eq:consisb} and ~\eqref{eq:regaux1}, we bound the consistency error functional as
\begin{align*}
&\quad\abs{a_h(u^\eps,v_h)-\dual{f}{v_h}
	+a_h(u_h,v^\eps)-\dual{g}{u_h}}\\
	&\le C(\eps+\eps/h)\Lr{\nm{\na u_0}{H^1(\Om)}+\nm{f}{L^2(\Om)}}\nm{g}{L^2(\Om)}.
\end{align*}
A combination of the above three estimates yields~\eqref{eq:l2errold}. 

For $m\ge 2, d=2$, noting that $A=A^t$ and the shift estimate~\eqref{eq:regaux1} is also valid for $u_0$, this gives~\eqref{eq:l2err1}.
	
For $m\geq 2$ and $d=3$,  $\chi$ is unbounded. Replacing~\eqref{eq:regaux1},~\eqref{eq:H1errb} and~\eqref{eq:consisb} by~\eqref{eq:regaux}, ~\eqref{eq:H1err} and~\eqref{eq:consis}, respectively, and proceeding along the same line that leads to~\eqref{eq:l2errold}, we obtain
\[
\nm{u-u_h}{L^{3/2}(\Om)}\le C(\eps+h^2+\eps/h)\Lr{\nm{\na u_0}{W^{1,3}(\Om)}+\nm{f}{L^3(\Om)}}.
\]
Noting that $A^t=A$ and the shift estimate~\eqref{eq:regaux} is also valid for $u_0$, this gives~\eqref{eq:l2err}.
\end{proof}
\subsection{Error estimates for MsFEM without oversampling}
We visit the error estimates of MsFEM without oversampling~\cite{HouWu:1997}. The multiscale basis function is $\phi^{\beta}=\{\phi_i^{\beta}\}_{i=1}^{d+1}$ is constructed as~\eqref{eq:overpro} with $S(\tau)$ replaced by $\tau$. 
%The multiscale finite element space is defined by
\[
V_h:=\text{Span}\{\phi_i \quad\text{for all nodes\;} x_i \text{\;of\;}\mc{T}_h\},
\]
and 
\(
V_h^0{:}=\set{v\in V_h}{v=0\quad \text{on~} \pa\Om}.
\)
The approximation problem reads as: Find $u_h\in V_h^0$ such that
\begin{equation}\label{eq:msfem}
a(u_h,v)=\dual{f}{v}\qquad\text{for all\quad}v\in V_h^0.
\end{equation}

Under the same assumptions of Theorem~\ref{thm:main} except that $A$ is not necessarily symmetric when $m\ge 2$, we prove the energy error estimate for MsFEM without oversampling. 
\begin{theorem}\label{thm:mainorig}
Assume $A$ is $1-$periodic and satisfies the Legendre-Hadamard condition~\eqref{eq:ellp}. Let $\Om$ be a bounded Lipschitz domain in $\mb{R}^d$. Let $u^{\eps}$ and $u_h$ be the solutions of~\eqref{eq:bvpweak} and ~\eqref{eq:msfem}, respectively.
	
For $m=1, d=2,3$ or $m\ge 2,d=2$, if $u_0\in H^2(\Om;\mb{R}^m)$, then		
\begin{equation}\label{eq:H1errb-ms}
\nm{\na(u^\eps-u_h)}{L^2(\Omega)}\le C\Lr{(\sqrt\eps+h)\nm{\na u_0}{H^1(\Om)}+\sqrt{\eps/h}\nm{\na u_0}{L^2(\Om)}},
\end{equation}
where $C$ depends on $\lam,\Lam,\Om$ and the mesh parameters $\sigma_0$ and $\sigma_1$.
	
For $m\geq2$ and $d=3$, if $u_0\in W^{2,3}(\Om;\mb{R}^m)$, then
\begin{equation}\label{eq:H1err-ms}
\nm{\na(u^\eps-u_h)}{L^2(\Omega)}\le C\Lr{(\sqrt\eps+h)\nm{\na u_0}{W^{1,3}(\Om)}+\sqrt{\eps/h}\nm{\na u_0}{L^2(\Om)}},
\end{equation}
where $C$ depends on $\lam,\Lam, \Om$ and the mesh parameters $\sigma_0$ and $\sigma_1$.
\end{theorem}

As a direct consequence of the above theorem, we obtain the L$^{d/(d-1)}$ error estimate for MsFEM without oversampling. The proof follows the same line that leads to Theorem~\ref{thm:l2err}, we omit the proof.
\begin{coro}
Under the same assumption of Theorem~\ref{thm:l2err} except that $A$ is not necessarily symmetric for $m\ge 2$. Let $u^{\eps}$ and $u_h$ be the solutions of~\eqref{eq:bvpweak} and ~\eqref{eq:msfem}, respectively. For $m=1,d=2,3$ or $m\ge 2,d=2$, there holds
\[
\nm{u-u_h}{L^2(\Om)}\le C(\eps+h^2+\eps/h)\nm{\na u_0}{H^1(\Om)}.
\]
	
For $m\ge 2$ and $d=3$, there holds
\[
\nm{u-u_h}{L^{3/2}(\Om)}\le C(\eps+h^2+\eps/h)\nm{\na u_0}{W^{1,3}(\Om)}.
\]
\end{coro}

The proof of Theorem~\ref{thm:mainorig} relies on Theorem~\ref{thm:rate} and Lemma~\ref{lemma:Mulplier}. We only sketch the main steps because the details are the same with the line leading to Theorem~\ref{thm:main}.
\vskip .3cm
\noindent
{\em Proof of Theorem~\ref{thm:mainorig}\;}
Noting that MsFEM without oversampling is conforming, i.e., $V_h^0\subset H_0^1(\Om;\mb{R}^m)$, we obtain 
\begin{equation}\label{eq:cea}
\nm{\na(u^\eps-u_h)}{L^2(\Omega)}\le(1+\Lam/\lam)\inf_{v\in V_h^0}\nm{\na(u^\eps-v)}{L^2(\Omega)}.
\end{equation}
	
Define MsFEM interpolant $\wt{u}(x)$ as~\eqref{eq:mlinterpolant}. Using the triangle inequality, we obtain
\[
\nm{\na(u^\eps-\wt{u})}{L^2(\Om)}\le\nm{\na(u^\eps-u_1^\eps)}{L^2(\Om)}+\nm{\na(\wt{u}-\wt{u}_1^\eps)}{L^2(\Om)}+\nm{\na(u_1^{\eps}-\wt{u}_1^\eps)}{L^2(\Om)}.
\]

The estimate of $\nm{\na(u^\eps-u_1^\eps)}{L^2(\Om)}$ follows from Theorem~\ref{thm:rate}, and the estimate of $\nm{\na(u_1^{\eps}-\wt{u}_1^\eps)}{L^2(\Om)}$ is the same with the corresponding term in Lemma~\ref{lema:app}. Note that $\wt{u}_1^\eps$ is the first order approximation of $\wt{u}$ over $\tau$. For $m=1,d=2,3$ or $m\ge 2,d=2$, using~\eqref{eq:1strateb}, we get
\begin{align*}
\nm{\na(\wt{u}-\wt{u}_1^\eps)}{L^2(\tau)}&\le C\sqrt{\eps/h_\tau}\nm{\na\pi u_0}{L^2(\tau)}\\
&\le C\Lr{\sqrt{\eps/h_\tau}\nm{\na u_0}{L^2(\tau)}+\sqrt{\eps h_{\tau}}\nm{\na u_0}{H^1(\tau)}}.
\end{align*}
Summing up the above estimate for all $\tau\in\mc{T}_h$, and using the inverse assumption of $\mc{T}_h$, we obtain 
\begin{equation}\label{eq:inter-new}
\nm{\na(\wt{u}-\wt{u}_1^\eps)}{L^2(\Om)}\le C\Lr{\sqrt{\eps/h}\nm{\na u_0}{L^2(\Om)}
+\sqrt{\eps h}\nm{\na u_0}{H^1(\Om)}}.
\end{equation}

For $m\ge 2$ and $d=3$, using~\eqref{eq:1stratea} and the fact that $\na\pi u_0$ is a piecewise constant matrix over $\tau$, we get
\[
\nm{\na(\wt{u}-\wt{u}_1^\eps)}{L^2(\tau)}\le C\sqrt{\eps/h_\tau}\abs{\tau}^{1/6}\nm{\na\pi u_0}{L^3(\tau)}
=C\sqrt{\eps/h_\tau}\nm{\na\pi u_0}{L^2(\tau)}.
\]
Proceeding along the same line that leads to~\eqref{eq:inter-new}, we obtain
\[
\nm{\na(\wt{u}-\wt{u}_1^\eps)}{L^2(\Om)}\le C\Lr{\sqrt{\eps/h}\nm{\na u_0}{L^2(\Om)}
+\sqrt{\eps h}\nm{\na u_0}{H^1(\Om)}}.
\]

A combination of all the above estimates completes the proof.\qed
\begin{remark}
We have used Theorem~\ref{thm:rate} to bound $\nm{\na(\wt{u}-\wt{u}_1^\eps)}{L^2(\tau)}$ instead of Lemma~\ref{lema:localhomo}, we need not assume the symmetry of $A$ when $m\geq2$.% for MsFEM without oversampling.
\end{remark}
\section{Conclusion}
Under suitable regularity assumptions on the homogenized solution, we proved the optimal energy error estimates for MsFEM with or without oversampling applying to elliptic systems with bounded measurable periodic coefficients. The present work may be extended to elliptic system with locally periodic coefficients, i.e., $A^\eps=A(x,x/\eps)$ with the aid of a new local multiplier estimate. The extension to elliptic system for the coefficients with stratified structure is also very interesting. We believe that the machineries developed in the present work may be useful to analyze other MsFEM such as the mixed MsFEM~\cite{ChenHou:2003}, Crouzeix-Raviart MsFEM~\cite{LeBrisLegoll:2014}, or MsFEM with different oversampling techniques~\cite{EfendievHouWu:2000}. We shall leave these for further pursuit.
\section*{Acknowledgments}
We gratefully acknowledge the helpful suggestions made by the anonymous referees, which greatly improved the presentation of the paper.
\bibliographystyle{siamplain}
\bibliography{msfem}
\end{document}